\documentclass[reqno]{amsart}
\usepackage{CJK}
\usepackage{hyperref}
\usepackage{cite}
\usepackage{amsmath}
\usepackage{mathrsfs}

\numberwithin{equation}{section}

\newtheorem{theorem}{Theorem}[section]
\newtheorem{lemma}[theorem]{Lemma}
\newtheorem{definition}[theorem]{Definition}

\newtheorem{proposition}[theorem]{Proposition}
\newtheorem{remark}[theorem]{Remark}

\allowdisplaybreaks

\newcommand{ \mint }{ {\int\hspace{-0.38cm}-}}

\begin{document}
	
	\title[\hfil H\"{o}lder regularity for mixed local and nonlocal equations\dots] {H\"{o}lder regularity for mixed local and nonlocal $p$-Laplace parabolic equations}
	
	\author[B. Shang and C. Zhang  \hfil \hfilneg]
	{Bin Shang and Chao Zhang$^*$}
	
	\thanks{$^*$Corresponding author.}
	
	\address{Bin Shang \hfill\break
		School of Mathematics, Harbin Institute of Technology,
		Harbin 150001, P.R. China} \email{shangbin0521@163.com}
	
	\address{Chao Zhang\hfill\break
		School of Mathematics and Institute for Advanced Study in Mathematics, Harbin Institute of Technology,
		Harbin 150001, P.R. China} \email{czhangmath@hit.edu.cn}

	.

	\maketitle

	\begin{abstract}
		We give a unified proof of H\"{o}lder regularity of weak solutions for mixed local and nonlocal $p$-Laplace type parabolic equations with the full range of exponents  $1<p<\infty$. Our proof is based on the expansion of positivity together with the energy estimate and De Giorgi type lemma.
	\end{abstract}
	
	\section{Introduction}
	\label{sec1}
	\par
	
	In this paper, we study the regularity results for the following mixed parabolic problem
	\begin{align}
	\label{1.1}
	\partial_t u(x,t)-\Delta_p u(x,t)+\mathcal{L}u(x,t)=0 \quad \text{in } E_T,\quad  1<p<\infty,
	\end{align}
	where $E_T:=E\times(0,T)$ with $T>0$ and $E$ being an open set in $\mathbb{R}^N$. The local $p$-Laplace operator is given by
	\begin{align*}
	-\Delta_p u:=-\mathrm{div}(|\nabla u|^{p-2}\nabla u).
	\end{align*}
	The nonlocal $p$-Laplace operator is defined as
	\begin{align}
	\label{1.2}
	\mathcal{L}u(x,t)=\mathrm{P.V.}\int_{\mathbb{R}^N}K(x,y,t)|u(x,t)-u(y,t)|^{p-2}(u(x,t)-u(y,t))\,dy,
	\end{align}
	where $\mathrm{P.V.}$ denotes the Cauchy principle value. The kernel $K$ satisfies that $K(x,y,t)=K(y,x,t)$ and
	\begin{align}
	\label{1.3}
	\frac{\Lambda^{-1}}{|x-y|^{N+sp}}\leq K(x,y,t)\leq \frac{\Lambda}{|x-y|^{N+sp}}
	\end{align}
	with $\Lambda\geq1$ and $0<s<1$ for every $x,y\in\mathbb{R}^N$ and $t\in(0,T)$. This kind of problems, which were extensively applied in plasma physics\cite{Bd13} and biology\cite{DLV21}, arise from the mixture of classical random walk and L\'{e}vy flight.
	
	Before giving our main contribution, let us review some known results. In the local setting, the H\"{o}lder regularity of weak solutions to the usual $p$-Laplace equation
	\begin{equation}
	\label{1.4}
	\partial_t u-\mathrm{div}(|\nabla u|^{p-2}\nabla u)=0
	\end{equation}
	was proved  by DiBenedetto\cite{D86} for $p>2$ and  by Chen-DiBenedetto\cite{CD88,CD89} for $1<p<2$. A comprehensive treatment of this issue goes back to the monograph \cite{D93} in which the  intrinsic scaling method plays a central role. Roughly speaking, the main idea of this method is to make the solutions of equation \eqref{1.4} behave like the solutions of heat equation over the intrinsic parabolic cylinders. It is worth mentioning that a more simple approach so-called expansion of positivity was introduced by DiBenedetto-Gianazza in \cite{DG08} and by DiBenedetto-Gianazza-Vespri in the monograph \cite{DGV12} to  access the theory of regularity for quasilinear parabolic equations. Based on expansion of positivity, Liao \cite{L20}  derived the H\"{o}lder continuity of locally bounded solutions to parabolic equations with structures modeled after the parabolic $p$-Laplace equation and the porous medium equation. We refer the readers to \cite{AM07,BDL21,DF85,M13,M02,W86} and references therein for more related results.

	For the nonlocal $p$-Laplace equation
	\begin{align}
	\label{1.5}
	\partial_t u(x,t)+\mathcal{L}u(x,t)=0,
	\end{align}
	where $\mathcal{L}$ is defined as in \eqref{1.2}, Str\"{o}mqvist\cite{Str19} showed the existence and boundedness of weak solutions for the case $p>2$ with the kernel $K(x,y,t)$ satisfying \eqref{1.3}. In view of iterated discrete differentiation of the equation in the spirit of Moser's technique, H\"{o}lder estimates with specific exponents of weak solutions in the case $p>2$
	was established by Brasco-Lindgren-Str\"{o}mqvist in\cite{BLS21}. Regarding the inhomogeneous case of \eqref{1.5}, Ding-Zhang-Zhou \cite{DZZ21}  proved the local boundedness ($1<p<\infty$) and H\"{o}lder regularity ($2<p<\infty$) for the local weak solutions by using the De Giorgi-Nash-Moser iteration. Very recently, the  H\"{o}lder regularity of weak solutions of \eqref{1.5}  in the full range $1<p<\infty$ was investigated by Liao \cite{L22}   without using any logarithmic estimate and any comparison principle. At the same time, Adimurthi-Prasad-Tewary \cite{APT22} obtained the same result by employing the expansion of positivity and provided a new form of isoperimetric inequality. More results can be found in \cite{CCV11,CD14,KS14,FK13}.
	
	In recent years, the mixed local and nonlocal problems have attracted increasing interest. For the special scenario
	\begin{align}
	\label{1.6}
	-\Delta u+(-\Delta)^s u=0,
	\end{align}
	Foondun\cite{F09} obtained the Harnack inequality and H\"{o}lder continuity by probabilistic methods,  see also \cite{CKSV12} for an another argument. Some qualitative and quantitative properties such as the maximum principle, interior Sobolev regularity and boundary regularity were developed by Biagi-Dipierro-Valdinoci-Vecchi in \cite{BDVV21}, where the authors utilized difference quotient and constructed barriers appropriately. Concerning the nonlinear framework of \eqref{1.6}
	\begin{align*}
	-\Delta_p u+(-\Delta)^s_p u=0,
	\end{align*}
	Garain-Kinnunen\cite{GK21} first studied the regularity properties of weak solutions incorporating local boundedness, H\"{o}lder continuity and Harnack estimates along with lower semicontinuity of weak supersolutions. When it comes to the parabolic analogue of \eqref{1.6}, Garain-Kinnunen\cite{GK} worked with sign changing supersolution and investigated the Harnack inequality with a tail term. A priori H\"{o}lder estimate, parabolic Harnack principle and heat kernel estimates were established in \cite{CK10}. Furthermore, Fang-Shang-Zhang \cite{FSZ21} proved the local boundedness and H\"{o}lder continuity of weak solutions to \eqref{1.1} for $1<p<\infty$ and $2<p<\infty$, respectively. We remark that, for the technical reason, it is necessary to impose the assumption that $p>2$ in order to obtain the Logarithmic type estimate. As a consequence, this restriction prevents an extension of H\"{o}lder continuity in \cite{DZZ21,FSZ21} to the subquadratic case $1<p<2$.

	The present paper is a continuation of our previous work \cite{FSZ21} to derive the H\"{o}lder continuity of weak solutions to \eqref{1.1} for  the subquadratic case that $1<p<2$.  Motivated by the ideas developed in a series of papers \cite{DG08,DGV12,L20,L22}, we will adopt the methods of expansion of positivity with strong geometric characteristics to obtain the regularity results of weak solutions in the full range of exponents $1<p<\infty$. We employ the Caccioppoli type inequality, De Giorgi's iteration without using Logarithmic type estimates and exponential change of variables. To overcome the difficulties arising from the nonlocal operator, the tail term (see \eqref{1.7} below) and the information of solutions will appear via an either-or form. We will estimate the tail term carefully so that it can be controlled by the local oscillation, which allows us to utilize the measure information and then to get the reduction of oscillation. Notice that our results also hold for mixed local and nonlocal $p$-Laplace parabolic equations with more general structure conditions.
	
	\smallskip
	Before stating the notations of weak solutions to \eqref{1.1}, we introduce the tail space
	\begin{align*}
	L_\alpha^q(\mathbb{R}^{N}):=\left\{v \in L_{\rm{loc}}^q(\mathbb{R}^{N}): \int_{\mathbb{R}^N} \frac{|v(x)|^q}{1+|x|^{N+\alpha}}\,dx<+\infty\right\}, \quad q>0 \text{ and } \alpha>0.
	\end{align*}
	
	Define the nonlocal tail by
	\begin{align}
	\label{1.7}
	\mathrm{Tail}(v;x_o,R;t_o-S,t_o):=\operatorname*{ess\,\sup}_{t_o-S <t<t_o}\left(R^p \int_{\mathbb{R}^{N} \backslash B_R(x_o)} \frac{|v(x, t)|^{p-1}}{|x-x_o|^{N+sp}}\,dx\right)^{\frac{1}{p-1}}.
	\end{align}
	If $x_o=0$, we denote tail as $\mathrm{Tail}(v;R;t_o-S,t_o)$.
	It is obvious that $\mathrm{Tail}(v;x_o,R;t_o-S,t_o)$ is well-defined for any $v\in L^\infty(t_o-S,t_o;L_{sp}^{p-1}(\mathbb{R}^N))$.

	Next, we present the definition of weak solutions to \eqref{1.1} as follows.
	
	\begin{definition}
		\label{def-1-1}
		A function $u$ is a weak subsolution (super-) to equation  \eqref{1.1} if 	$u\in L^p_{\rm{loc}}(0,T;W_{\rm {loc}}^{1,p}(E)) \cap C_{\rm{loc}}(0,T;L_{\rm {loc}}^2(E))\cap L^\infty(0,T;L_{sp}^{p-1}(\mathbb{R}^N))$ for any closed interval $[t_1, t_2] \subseteq(0,T)$ and any compact subset $K\subset E$, there holds that
		\begin{align}	\label{1.8}
		&\int_K u(x,t_2)\zeta(x,t_2)\,dx-\int_K u(x,t_1)\zeta(x,t_1)\,dx
		-\int_{t_1}^{t_2} \int_K u(x,t)\partial_t \zeta(x,t)\,dxdt\nonumber\\
		&\quad+\int_{t_1}^{t_2}\int_K|\nabla u|^{p-2}\nabla u\cdot \nabla\zeta\,dxdt +\int_{t_1}^{t_2} \mathcal{E}(u,\zeta,t)\,dt \leq(\geq) 0,
		\end{align}
		for every nonnegative test function $\zeta \in L^p_{\rm{loc}} (0,T;W^{1,p}_0(K))\cap W^{1,2}_{\rm{loc}}(0,T;L^2(K))$, where
		\begin{align*}
		\mathcal{E}(u,\zeta,t):=\frac{1}{2} \int_{\mathbb{R}^N} \int_{\mathbb{R}^N}&\Big[|u(x,t)-u(y,t)|^{p-2}(u(x,t)-u(y,t))\\
		&\times (\zeta(x,t)-\zeta(y,t))K(x,y,t)\Big]\,dxdy.
		\end{align*}
		A function $u$ is a local weak solution to \eqref{1.1} if and only if $u$ is a local weak subsolution and supersolution.
	\end{definition}

	We now are in position to state our main result as follows.
	
	\begin{theorem}
		\label{thm-1-2}
		Let $1<p<\infty$. Assume that $u$ is a locally bounded, local weak solution to equation \eqref{1.1} in $E_T$. Then $u$ is locally H\"{o}lder continuous in $E_T$. More precisely,  there exist constants $\gamma>1$ and $\beta\in (0,1)$ depending a prior on $N,p,s,\Lambda$ such that for any $0<\rho<R<\widetilde{R}\leq 1$, the cylinders $(x_o,t_o)+Q_\rho(\omega^{2-p}):=B_\rho(x_o)\times(t_o-\omega^{2-p}\rho^p,t_o)\subset(x_o,t_o)+Q_{\widetilde{R}}:=B_{\widetilde{R}}\times(t_o-\widetilde{R}^p,t_o)$ are included in $E_T$,
		it holds that	
		\begin{align*}
		\operatorname*{ess\,osc}_{(x_o,t_o)+Q_\rho(\omega^{2-p})} u\leq\gamma\omega\left(\frac{\rho}{R}\right)^\beta,
		\end{align*}
		where
		\begin{align*}
		\omega=2 \operatorname*{ess\,\sup}_{(x_o,t_o)+Q_{\widetilde{R}}}|u|+\mathrm{Tail}(u;x_o,\widetilde{R};t_o-\widetilde{R}^p,t_o).
		\end{align*}	
	\end{theorem}
	
	\begin{remark}
		In Theorem \ref{thm-1-2}, we assume that the local weak solution $u$ is locally bounded. In fact, the local boundedness of weak solutions has been proved  in our previous paper \cite[Theorems 1.2 and 1.3]{FSZ21}. We also stress that Theorem \ref{thm-1-2} still holds for more general jumping kernel $K(x, y, t)$ such as
		\begin{align*}
		K(x,y,t):=\frac{1}{|x-y|^{N+s_1 p}}  \textbf{{\rm 1}}_{\{|x-y| \leq 1\}}+\frac{1}{|x-y|^{N+s_2 p}} \textbf{{\rm 1}}_{\{|x-y|>1\}},
		\end{align*}
		where $s_1,s_2\in(0,1).$
		
	\end{remark}
	
	This paper is organized as follows. In Section \ref{sec2}, we fix basic notations and present several auxiliary lemmas. Section \ref{sec3} is devoted to establishing some lemmas which are crucial to obtain the H\"{o}lder regularity. We will give the proof of Theorem \ref{thm-1-2} with $1<p\leq2$  in Section \ref{sec4}. Finally, the other case $p>2$ of Theorem \ref{thm-1-2} will be proved in Section \ref{sec5}.

	\section{Preliminaries}
	\label{sec2}
	
	In this section, we first introduce some notations and then give several useful lemmas.
	
	\subsection{Notation}
	
	For convenience, we give some notations to be used later. Let $B_\rho(y)$ be the open ball centered at $x_o\in\mathbb{R}^N$ with radius $\rho>0$. For fixed $(x_o, t_o)\in\mathbb{R}^N\times\mathbb{R}$, we define the forward and backward parabolic cylinders as
	$$
	(x_o, t_o)+Q_{\rho}^{+}(\theta)=B_{\rho}(x_o) \times\left(t_o, t_o+\theta \rho^{p}\right]
	$$
	and
	$$
	(x_o,t_o)+Q_{\rho}^{-}(\theta)=B_{\rho}(x_o) \times\left(t_o-\theta \rho^{p}, t_o\right].
	$$
	For simplicity, we will omit $(x_o,t_o)$ from the definitions above. We also omit $\theta$ when $\theta=1$. In addition, define the backward parabolic cylinders
	$$
	(x_o,t_o)+Q(R,S)=B_R(x_o) \times\left(t_o-S, t_o\right].
	$$
	
	Set
	\begin{align*}
	d\mu=d\mu(x,y,t)=K(x,y,t)\,dxdy.
	\end{align*}
	
	For fixed $k\in \mathbb{R}$, let
	\begin{align*}
	(u-k)_+=\max\{u-k,0\} \quad \textmd{and} \quad  (u-k)_-=\max\{-(u-k),0\}.
	\end{align*}

	Throughout the paper, we always assume that $0<\rho<R<\widetilde{R}\leq 1$.
	
	\subsection{Auxiliary lemmas}
	\begin{lemma} [Lemma 2.3, \cite{FSZ21}]
		\label{lem-2-1}
		Let $0<t_1<t_2$ and $p\in(1,\infty)$. Then for every
		\begin{align*}
		u\in L^p\left(t_1,t_2;W^{1,p}(B_\rho)\right) \cap L^\infty\left(t_1,t_2;L^2(B_\rho)\right),
		\end{align*}
		it holds that
		\begin{align*}
		&\quad\int_{t_1}^{t_2} \mint_{B_\rho}|u(x,t)|^{p(1+\frac{2}{N})}\,dxdt \nonumber\\
		&\leq C\left(\rho^p \int_{t_1}^{t_2}\mint_{B_\rho} |\nabla u(x,t)|^p\,dxdt+\int_{t_1}^{t_2} \mint_{B_\rho}|u(x,t)|^p\,dxdt\right) \nonumber\\
		&\quad \times \left(\operatorname*{ess\,\sup}_{t_1<t<t_2} \mint_{B_\rho}|u(x, t)|^2\,dx\right)^{\frac{p}{N}},
		\end{align*}
		where $C>0$ only depends on $p$ and $N$.
	\end{lemma}

	\begin{lemma} [Lemma 4.1, \cite{D93}]
		\label{lem-2-2}
		Let $\{Y_j\}_{j=0}^\infty$ be a sequence of positive numbers such that
		\begin{align*}
		Y_{j+1}\leq Kb^jY_j^{1+\delta},\quad j=0,1,2, \ldots
		\end{align*}
		for some constants $K$, $b>1$ and $\delta>0$. If
		\begin{align*}
		Y_o\leq K^{-\frac{1}{\delta}}b^{-\frac{1}{\delta^2}},
		\end{align*}
		then we have $\lim_{j\rightarrow\infty}Y_j=0$.
	\end{lemma}
	
	\section{crucial results}
	\label{sec3}

	In this section, several preliminary results will be discussed.  We first present the following Caccioppoli type estimate for weak subsolutions of \eqref{1.1} and the supersolution case is similar. 
	\begin{proposition} [Caccioppoli-type inequality]
		\label{lem-3-1}
		Assume that $p>1$ and $u$ is a local weak subsolution to \eqref{1.1}. Let cylinders $Q(R,S)\subset E_T$. For any piecewise smooth cutoff functions $\zeta(x,t)\in [0,1]$ vanishing on $\partial B_\rho$ and $\zeta(x,t)=0$ for $t<t_o-S$ it holds that
		\begin{align}
		\label{3.1}
		&\mathop{\mathrm{ess}\,\sup}_{t_o-s<t<t_o}\int_{B_R} w_+^2\zeta^p(x,t)\,dx-\int_{B_R}w_+^2\zeta^p(x,t_o-S)\,dx
		+\iint_{Q(R,S)}|\nabla w_+(x,t)|^p\zeta^p(x,t)\,dxd\tau\nonumber\\
		&+\int_{t_o-S}^{t_o}\int_{B_R} \int_{B_R} \min\{\zeta(x,t),\zeta(y,t)\}^p \frac{|w_+(x,t)-w_+(y,t)|^p}{|x-y|^{N+sp}}\,dxdyd\tau\nonumber\\
		&+\iint_{Q(R,S)} w_+\zeta^p(x,t)\int_{B_R} \frac{w_-^{p-1}(y, t)}{|x-y|^{N+sp}}\,dydxd\tau \nonumber\\
		\leq&\gamma\iint_{Q(R,S)}w_+^p(x,t)|\nabla\zeta(x,t)|^p+w_+^2(x,t)|\partial_t \zeta^p|\,dxd\tau\nonumber\\
		&+\gamma \int_{t_o-S}^{t_o}\int_{B_R}\int_{B_{R}}\max\{w_+(x,t), w_+(y,t)\}^p \frac{|\zeta(x,t)-\zeta(y,t)|^p}{|x-y|^{N+sp}}\,dxdyd\tau \nonumber\\
		&+\gamma\mathop{\mathrm{ess}\,\sup}_{\stackrel{t_o-S<t<t_o}{x\in \mathrm{supp}\,\zeta(\cdot,t)}}\int_{\mathbb{R}^N \backslash B_R}\frac{w_+^{p-1}(y,t)}{|x-y|^{N+sp}}\,dy\iint_{Q(R,S)}w_+\zeta^p(x,t)\,dxd\tau,
		\end{align}
		where $w_+=(u-k)_+$ with $k\in\mathbb{R}$ and $\gamma>0$ depends only on $N,p,s,\Lambda$.
	\end{proposition}
	
	\begin{remark}
		The Caccioppoli inequality above can be obtained by using the method of Lemma 3.1 in our previous paper \cite{FSZ21}. However, we follow the estimates (2.1)--(2.2) in \cite{L22} here to get the additional term
		$$
		\iint_{Q(R,S)} w_+\zeta^p(x,t)\int_{B_R} \frac{w_-^{p-1}(y, t)}{|x-y|^{N+sp}}\,dydxd\tau.
		$$
	\end{remark}

	From now on, we always denote $\mathcal{Q}=B_R(x_o)\times(T_1,T_2]\subset E_T$. The number $\mu^\pm$ and $\omega$ satisfy
	\begin{align*}
	\mu^+\geq \operatorname*{ess\,\sup}_{\mathcal{Q}} u, \quad \mu^-\leq \operatorname*{ess\,\inf}_{\mathcal{Q}} u, \quad \omega\geq\mu^+-\mu^-.
	\end{align*}
	
	We now prove a De Giorgi type lemma for weak supersolutions. Notice that the tail will appear in terms of either-or in order to control the nonlocal terms.
	
	\begin{lemma}
		\label{lem-3-2}	
		Let $1<p<\infty$. Suppose that $u$ is a locally bounded, local weak supersolution to \eqref{1.1} in $E_T$. For parameters $\delta,\xi\in(0,1)$, let $\theta=\delta(\xi\omega)^{2-p}$ and $Q_\rho^-(\theta)\subset\mathcal{Q}$. There exists a constant $\nu\in(0,1)$ depending on $N,p,s,\Lambda$ and $\delta$ such that if
		\begin{align}
		\left|[u-\mu_-\leq \xi\omega] \cap Q_\rho^-(\theta)\right| \leq \nu|Q_\rho^-(\theta)|,
		\end{align}
		then either
		\begin{align*}
		\left(\frac{\rho}{R}\right)^{\frac{p}{p-1}} \mathrm{Tail}\left((u-\mu^-)_-;x
		_o,R;T_1,T_2\right)>\xi\omega,
		\end{align*}
		or
		\begin{align*}
		u-\mu^-\geq \frac{1}{2}\xi\omega\quad\text {a.e. in } Q_{\frac{1}{2}\rho}^{-}(\theta).
		\end{align*}
	\end{lemma}
	\begin{proof}
		We can assume $(x_o,t_o)=(0,0)$ and $\mu^-=0$ . Let
		\begin{align*}
		\rho_j=\frac{\rho}{2}+\frac{\rho}{2^j},\quad \tilde{\rho}_j=\frac{\rho_j+\rho_{j+1}}{2},\quad B_j=B_{\rho_j},\quad \widetilde{B}_j=B_{\tilde{\rho}_j},\quad  j=0,1,2\ldots.
		\end{align*}
		Denote the cylinders
		\begin{align*}
		Q_j=B_j \times\left(-\theta \rho_j^p, 0\right],\quad \widetilde{Q}_j=\widetilde{B}_j \times\left(-\theta \tilde{\rho}_j^p, 0\right],\quad j=0,1,2\ldots.
		\end{align*}
		Let the level
		\begin{align*}
		k_j=\frac{\xi\omega}{2}+\frac{\xi\omega}{2^{j+1}}.
		\end{align*}
		
		We apply the energy estimate \eqref{3.1} for function $(u-k_j)_-$ over the cylinder $Q_j$. Choose the cutoff function $\zeta(x,t)=\zeta_1(x)\zeta_2(t)$ in $Q_j$ satisfying that
		\begin{align*}
		0\leq \zeta_1(x)\leq 1,\quad |\nabla\zeta_1(x)|\leq\frac{2^{j+1}}{\rho},\quad \zeta_1(x)=1 \text{ in } B_{j+1},\quad \zeta_1(x)=0 \text{ in } \mathbb{R}^N\backslash\widetilde{B}_j
		\end{align*}
		and
		\begin{align*}
		0\leq\zeta_2(t)\leq 1,\quad 0\leq \zeta_{2,t}\leq\frac{2^{p(j+1)}}{\theta\rho^p},\quad \zeta_2(t)=1 \text{ for } t\geq-\theta\rho_{j+1}^p,\quad \zeta_2(t)=0 \text{ for } t<-\theta\tilde{\rho}_j^p.
		\end{align*}
		Then it yields that
		\begin{align*}
		&\operatorname*{ess\,\sup}_{-\theta\rho^p_j<t<0}\int_{B_j} (u-k_j)_-^2\zeta^p(x,t)\,dx+\iint_{Q_j}|\nabla(u-k_j)_-\zeta(x,t)|^p\,dxd\tau\nonumber\\
		&\quad+\int_{-\theta\rho_j^p}^0\int_{B_j}\int_{B_j}|(u-k_j)_-\zeta(x,t)-(u-k_j)_-\zeta(y,t)|^p\,d\mu d\tau\nonumber\\
		&\leq \gamma\frac{2^{jp}}{\rho^p}\iint_{Q_j}(u-k_j)_-^p\,dxd\tau+\gamma\frac{2^{jp}}{\theta\rho^p} \iint_{Q_j}(u-k_j)_-^2\,dxd\tau\\
		&\quad+\gamma\int_{-\theta\rho_j^p}^0\int_{B_j}\int_{B_j}(\max\{(u-k_j)_-(x,t),(u-k_j)_-(y,t)\})^p|\zeta(x,t)-\zeta(y,t)|^p\,d\mu d\tau\nonumber\\
		&\quad+\gamma \mathop{\mathrm{ess}\,\sup}_{\stackrel{-\theta \rho_j^p<t<0}{(x,t)\in \mathrm{supp}\,\zeta}}\int_{\mathbb{R}^N \backslash B_j}\frac{(u-k_j)_-^{p-1}(y,t)}{|x-y|^{N+sp}}\,dy\iint_{Q_j}(u-k_j)_-\zeta^p(x,t)\,dxd\tau\nonumber\\
		&=: I_1+I_2+I_3+I_4.
		\end{align*}
		
		Using the definition of $k_j$, we estimate $I_1$ as
		\begin{align*}
		I_1 &\leq \gamma\frac{2^{jp}}{\rho^p} \iint_{Q_j}\left(u-k_j\right)_-^p\chi_{\{u(x,t)<k_j\}}\,dxd\tau\\
		&\leq \gamma\frac{2^{jp}}{\rho^p}(\xi\omega)^p\left|[u<k_j]\cap Q_j\right|.
		\end{align*}
		For the terms $I_2$ and $I_3$, it holds that
		\begin{align*}
		I_2\leq \gamma\frac{2^{jp}}{\theta\rho^p}(\xi\omega)^2\left|[u<k_j]\cap Q_j\right|
		\end{align*}
		and
		\begin{align*}
		I_3\leq \gamma\frac{2^{jp}}{\rho^p} (\xi\omega)^p\left|[u<k_j]\cap Q_j\right|.
		\end{align*}
		For every $x\in \mathrm{supp}\, \zeta_1$ and $y\in\mathbb{R}^N\backslash B_j$, notice that
		\begin{align*}
		\frac{1}{|x-y|}=\frac{1}{|y|}\frac{|x-(x-y)|}{|x-y|}\leq\frac{1+2^{j+3}}{|y|}\leq\frac{2^{j+4}}{|y|}.
		\end{align*}
		If we enforce
		\begin{align*}
		\left(\frac{\rho}{R}\right)^{\frac{p}{p-1}} \mathrm{Tail}\left(u_-;R;T_1,T_2\right)\leq\xi\omega,
		\end{align*}
		the term $I_4$ can be estimated as
		\begin{align*}
		I_4&\leq \gamma 2^{j(N+sp)}\xi\omega|[u<k_j]\cap Q_j|\left(\gamma 2^{spj}\frac{(\xi\omega)^{p-1}}{\rho^{sp}}+\mathop{\mathrm{ess}\,\sup}_{-\theta \rho_j^p<t<0}\int_{\mathbb{R}^N \backslash B_R}\frac{u_-^{p-1}(y,t)}{|y|^{N+sp}}\,dy\right)\\
		&\leq\gamma 2^{j(N+sp)}\frac{\xi\omega}{\rho^p}|[u<k_j]\cap Q_j|\left(\gamma 2^{spj}(\xi\omega)^{p-1}+\left(\frac{\rho}{R}\right)^p[\mathrm{Tail}(u_-;R;T_1,T_2)]^{p-1}\right)\\
		&\leq\gamma 2^{j(N+2sp)}\frac{(\xi\omega)^p}{\rho^p}|[u<k_j]\cap Q_j|.
		\end{align*}
		
		Combining the estimates $I_1$--$I_4$ and discarding the positive term on the left-hand side yields that
		\begin{align*}
		&\quad \operatorname*{ess\,\sup}_{-\theta\rho^p_{j+1}<t<0}\int_{B_{j+1}} (u-k_j)_-^2\,dx+\iint_{Q_{j+1}}|\nabla(u-k_j)_-|^p\,dxd\tau\\
		&\leq  \gamma\frac{2^{j(N+2sp+p)}}{\delta\rho^p}(\xi\omega)^p|[u<k_j]\cap Q_j|.
		\end{align*}
		According to the Sobolev inequality in Lemma \ref{lem-2-1}, it follows that
		\begin{align*}
		&\quad \iint_{Q_{j+1}}[\left(u-k_j\right)_-]^{p\frac{N+2}{N}}\,dxd\tau \\
		&\leq \gamma \iint_{Q_{j+1}}\left|\nabla[(u-k_j)_-]\right|^p\,dxd\tau \times\left(\operatorname*{ess\,\sup}_{-\theta\rho^p_j<t<0} \int_{B_{j+1}}[(u-k_j)_-]^2\,dx\right)^{\frac{p}{N}} \\
		&\leq \gamma\left[\frac{2^{j(N+2sp+p)}}{\delta\rho^p}(\xi\omega)^p\right]^{\frac{N+p}{N}}|[u<k_j]\cap Q_j|^{\frac{N+p}{N}}.
		\end{align*}
		On the other hand,
		\begin{align*}
		\iint_{Q_{j+1}}[(u-k_j)_-]^{p\frac{N+2}{N}}\,dxd\tau\geq\left(\frac{\xi\omega}{2^{j+2}}\right)^{p\frac{N+2}{N}}\left|\left[u<k_{j+1}\right] \cap Q_{j+1}\right|.
		\end{align*}

		Define
		\begin{align*}
		Y_j=\frac{\left|[u<k_j] \cap Q_j\right|}{\left|Q_j\right|}.
		\end{align*}
		Thus we obtain
		\begin{align*}
		Y_{j+1} \leq \gamma b^j\delta^{-1} Y_j^{1+\frac{p}{N}},
		\end{align*}
		where
		\begin{align*}
		b=2^{\frac{N+p}{N}(N+p+2sp)+p\frac{N+2}{N}}.
		\end{align*}
		From Lemma \ref{lem-2-2}, there holds that $Y_j\rightarrow0$ as $j\rightarrow\infty$ if
		\begin{align*}
		Y_o\leq \gamma^{-1} b^{-\frac{N^2}{p^2}}\delta^{-\frac{N}{p}} \stackrel{\text { def }}{=} \nu,
		\end{align*}
		which coincides with the assumption \eqref{3.1}.
	\end{proof}
	
	Next, we consider Lemma \ref{lem-3-2} on the cylinder $Q_{\rho}^+(\theta)$ and give the additional information on initial data.
	\begin{lemma}
		\label{lem-3-3}
		Assume that $u$ is a locally bounded, local weak supersolution to \eqref{1.1} in $E_T$. Let $p\in(1,\infty)$, and the parameter $\xi\in(0,1)$. There exists a constant $\nu_o$ only depending on $N,p,s,\Lambda$ such that if
		\begin{align*}
		u(\cdot,t_o)-\mu^-\geq\xi\omega \quad\text {a.e. in } B_\rho(x_o),
		\end{align*}
		then either
		\begin{align*}
		\left(\frac{\rho}{R}\right)^{\frac{p}{p-1}} \mathrm{Tail}\left((u-\mu^-)_-;x_o,R;T_1,T_2\right)>\xi\omega,
		\end{align*}
		or
		\begin{align*}
		u-\mu^-\geq\frac{1}{2}\xi\omega \quad\text {a.e. in } B_{\frac{1}{2}\rho}(x_o)\times(t_o,t_o+\nu_o(\xi\omega)^{2-p}\rho^p]
		\end{align*}
		with the cylinders being included in $\mathcal{Q}$.
		\begin{proof}
			Without loss of generality, let $(x_o,t_o)=(0,0)$ and $\mu^-=0$. We apply the energy estimate \eqref{3.1} on the cylinder $Q_\rho^+(\theta)$ for some $\theta$ will be determined later. Let $\rho_j,\tilde{\rho}_j,B_j,\widetilde{B}_j$ and $k_j$ define as Lemma \ref{lem-3-2}. Denote
			\begin{align*}
			Q_j=B_j\times(0,\theta\rho_j^p),\quad \widetilde{Q}_j=B_j\times(0,\theta\tilde{\rho}_j^p).
			\end{align*}
			Choose the piecewise smooth cutoff function $\zeta(x)$ independent of time which satisfies
			\begin{align*}
			|\nabla \zeta(x)|\leq\frac{2^{j+1}}{\rho}, \quad \zeta(x)=1 \text{ in } B_{j+1}, \quad \zeta(x)=0 \text{ in } \mathbb{R}^N\backslash\widetilde{B}_j.
			\end{align*}
			As a result, we have
			\begin{align}
			\label{3.3}
			& \quad \operatorname*{ess\,\sup}_{0<t<\theta\rho^p_{j+1}}\int_{B_{j+1}} (u-k_j)_-^2\,dx+\iint_{Q_{j+1}}|\nabla(u-k_j)_-|^p\,dxd\tau\nonumber\\
			&\leq \gamma\frac{2^{jp}}{\rho^p}\iint_{Q_j}(u-k_j)_-^p\,dxd\tau\nonumber\\
			&\quad+\gamma\int_0^{\theta\rho_j^p}\int_{B_j}\int_{B_j}(\max\{(u-k_j)_-(x,t),(u-k_j)_-(y,t)\})^p|\zeta(x,t)-\zeta(y,t)|^p\,d\mu d\tau\nonumber\\
			&\quad+\gamma \mathop{\mathrm{ess}\,\sup}_{\stackrel{0<t<\theta \rho_j^p}{x\in \mathrm{supp}\,\zeta}}\int_{\mathbb{R}^N \backslash B_j}\frac{(u-k_j)_-^{p-1}(y,t)}{|x-y|^{N+sp}}\,dy\iint_{Q_j}(u-k_j)_-\zeta^p(x,t)\,dxd\tau.
			\end{align}
			The right-hand side of \eqref{3.3} can be estimated as Lemma \ref{lem-3-2}, if we enforce
			\begin{align*}
			\left(\frac{\rho}{R}\right)^{\frac{p}{p-1}} \mathrm{Tail}\left(u_-;R;T_1,T_2\right)\leq\xi\omega,
			\end{align*}
			it follows that
			\begin{align*}
			&\operatorname*{ess\,\sup}_{0<t<\theta\rho^p_{j+1}}\int_{B_{j+1}} (u-k_j)_-^2\,dx+\iint_{Q_{j+1}}|\nabla(u-k_j)_-|^p\,dxd\tau\\
			\leq&\gamma\frac{2^{j(N+2sp+p)}}{\rho^p}(\xi\omega)^p|[u<k_j]\cap Q_j|.
			\end{align*}
			
			Let
			\begin{align*}
			Y_j=\frac{\left|[u<k_j] \cap Q_j\right|}{\left|Q_j\right|}.
			\end{align*}
			Following the same step of Lemma \ref{lem-3-2}, we can derive the recursive inequality
			\begin{align*}
			Y_{j+1}\leq\gamma b^j\left[\frac{\theta}{(\xi\omega)^{2-p}}\right]^{\frac{p}{N}} Y_j^{1+\frac{p}{N}},
			\end{align*}
			where $\gamma$ and $b$ depend only on $N,p,s,\Lambda$. Utilizing Lemma \ref{lem-2-2}, there exists a constant $\nu_o$ depending only on $N,p,s,\Lambda$ such that if
			\begin{align*}
			Y_o\leq \frac{\nu_o(\xi\omega)^{2-p}}{\theta},
			\end{align*}
			then $Y_j\rightarrow 0$. Finally, we choose $\theta=\nu_o(\xi\omega)^{2-p}$ to reach the conclusion.	
		\end{proof}
	\end{lemma}
	
	The expansion of positivity about the time variable will be given in the next lemma. The main idea of which is extending the measure information of positivity at some time level on the ball $B_\rho(x_o)$ to later time.
	
	\begin{lemma}
		\label{lem-3-4}
		Suppose that $1<p<\infty$ and $u$ is a locally bounded, local weak supersolution to \eqref{1.1} in $E_T$. Let $\xi,\alpha\in(0,1)$. There exist $\delta$ and $\epsilon$ in $(0,1)$ that depend on $N,p,s,\Lambda$ and $\alpha$, such that if
		\begin{align}
		\label{3.4}
		\left|[u(\cdot,t_o)-\mu^- \geq\xi\omega] \cap B_\rho(x_o)\right| \geq \alpha\left|B_\rho(x_o)\right|,
		\end{align}
		then either
		\begin{align*}
		\left(\frac{\rho}{R}\right)^{\frac{p}{p-1}} \mathrm{Tail}\left((u-\mu^-)_-;x_o,R;T_1,T_2\right)>\xi\omega,
		\end{align*}
		or
		\begin{align*}
		\left|[u(\cdot,t)-\mu^-\geq\epsilon \xi\omega] \cap B_\rho(x_o)\right| \geq \frac{\alpha}{2} \left|B_\rho(x_o)\right| \quad \text { for all } t\in\left(t_o,t_o+\delta(\xi\omega)^{2-p}\rho^p\right]
		\end{align*}
		with the cylinder being included in $\mathcal{Q}$.
		\begin{proof}
			We may assume $(x_o,t_o)=(0,0)$ and $\mu^-=0$. For $k>0$ and $t>0$, define
			\begin{align*}
			A_{k,\rho}(t)=[u(\cdot,t)<k] \cap B_{\rho}.
			\end{align*}
			Consider the energy estimate \eqref{3.1} for functions $(u-\xi\omega)_-$ over the cylinder $Q_\rho^+(\theta)$. Let the cutoff function $\zeta\in C_o^\infty\left(B_{\frac{\rho(2-\sigma)}{2}}\right)$ independent of $t$ satisfy $\zeta(x)=1$ in $B_{(1-\sigma)\rho}$ and $|\nabla\zeta|\leq\frac{1}{\sigma\rho}$, where $\sigma\in(0,1]$ will be chosen later. Enforcing
			\begin{align*}
			\left(\frac{\rho}{R}\right)^{\frac{p}{p-1}} \mathrm{Tail}\left(u_-;R,T_1,T_2\right)\leq\xi\omega,
			\end{align*}
			then we can deduce that
			\begin{align}
			\label{3.6}
			&\quad\int_{B_{(1-\sigma) \rho}}(u-\xi\omega)_-^2(x,t)\,dx \nonumber \\
			&\leq\int_{B_\rho}(u-\xi\omega)_-^2(x,0)\,dx
			+\frac{\gamma}{(\sigma\rho)^p} \int_{0}^{\theta \rho^p}\int_{B_\rho}(u-\xi\omega)_-^p\,dxd\tau\nonumber \\
			&\qquad+\gamma\int_{0}^{\theta \rho^p}\int_{B_\rho}\int_{B_\rho}(\max\{(u-\xi\omega)_-(x,t),(u-\xi\omega)_-(y,t)\})^p|\zeta(x)-\zeta(y)|^p\,d\mu d\tau\nonumber\\
			&\qquad+\gamma \mathop{\mathrm{ess}\,\sup}_{\stackrel{0<t<\theta \rho^p}{x \in \mathrm{supp}\,\zeta}}\int_{\mathbb{R}^N \backslash B_\rho}\frac{(u-\xi\omega)_-^{p-1}(y,t)}{|x-y|^{N+sp}}\,dy\int_{0}^{\theta \rho^p}\int_{B_\rho}(u-\xi\omega)_-\zeta^p(x)\,dxd\tau\nonumber\\
			&\leq \left[ (\xi\omega)^2(1-\alpha)+\gamma \frac{\theta (\xi\omega)^p}{\sigma^p}+\gamma\frac{\theta (\xi\omega)^p}{\sigma^{N+sp}}\right]\left|B_{\rho}\right|\nonumber \\
			&\leq (\xi\omega)^2\left[(1-\alpha)+C \frac{\theta (\xi\omega)^{p-2}}{\sigma^{N+p}}\right]\left|B_{\rho}\right|,
			\end{align}
			where in the second line from bottom of \eqref{3.6} we estimated it as Lemma \ref{lem-3-2} and used the assumption \eqref{3.4}. On the other hand, the term on the left-hand is evaluated as
			\begin{align}
			\label{3.7}
			\int_{B_{(1-\sigma) \rho}}(u-\xi\omega)_-^2(x,t)\,dx &\geq\int_{B_{(1-\sigma)\rho} \cap[u<\epsilon\xi\omega]}(u-\xi\omega)_-^2(x,t)\,dx\nonumber\\
			& \geq (\xi\omega)^2(1-\epsilon)^2\left|A_{\epsilon\xi\omega,(1-\sigma)\rho}(t)\right|,
			\end{align}
			where $\epsilon\in(0,1)$ will be chosen later.
			It is easy to check that
			\begin{align}
			\label{3.8}
			\left|A_{\epsilon \xi\omega,\rho}(t)\right| &=\left|A_{\epsilon \xi\omega,(1-\sigma) \rho}(t) \cup\left(A_{\epsilon \xi\omega, \rho}(t)-A_{\epsilon \xi\omega,(1-\sigma) \rho}(t)\right)\right|\nonumber \\
			& \leq\left|A_{\epsilon \xi\omega,(1-\sigma) \rho}(t)\right|+\left|B_{\rho}-B_{(1-\sigma) \rho}\right| \nonumber\\
			& \leq\left|A_{\epsilon \xi\omega,(1-\sigma) \rho}(t)\right|+N \sigma\left|B_{\rho}\right|.
			\end{align}
			In light of \eqref{3.6}--\eqref{3.8}, we have
			\begin{align*}
			\left|A_{\epsilon \xi\omega,\rho}(t)\right|&\leq\frac{1}{(\xi\omega)^2(1-\epsilon)^2} \int_{B_{(1-\sigma) \rho}}(u-\xi\omega)^2_-(x,t)\,dx+N\sigma\left|B_\rho\right| \\
			&\leq\frac{1}{(1-\epsilon)^2}\left[(1-\alpha)+C\frac{\theta (\xi\omega)^{p-2}}{\sigma^{N+p}} +N \sigma\right]\left|B_{\rho}\right|.
			\end{align*}
			
			Finally, we can choose  $\theta=\delta (\xi\omega)^{2-p}$ and
			\begin{align*}
			\sigma=\frac{\alpha}{8N}, \quad \epsilon \leq 1-\frac{\sqrt{1-\frac{3}{4} \alpha}}{\sqrt{1-\frac{1}{2} \alpha}},  \quad \delta=\frac{\alpha^{N+p+1}}{C2^{3(p+N)} N^{N+p}},
			\end{align*}
			such that
			\begin{align*}
			\left|A_{\epsilon \xi\omega,\rho}(t)\right|\leq\left(1-\frac{1}{2}\alpha\right)|B_\rho|,
			\end{align*}
			which implies the desired result.
		\end{proof}
	\end{lemma}

	The following lemma is about measure shrinking that involves the condition in each slice of time.
	
	\begin{lemma}
		\label{lem-3-5}
		Assume that $1<p<\infty$ and $u$ is a locally bounded, local weak supersolution to \eqref{1.1} in $E_T$. There exist some $\delta,\sigma$ and $\xi\in(0,\frac{1}{2})$ such that if
		\begin{align*}
		\left|[u(\cdot,t)-\mu^-\geq\xi\omega] \cap B_\rho(x_o)\right| \geq\alpha \left|B_\rho(x_o)\right| \quad \text { for all } t\in\left(t_o-\theta\rho^p,t_o\right]
		\end{align*}
		with $\theta=\delta(\sigma\xi\omega)^{2-p}$, then either
		\begin{align*}
		\left(\frac{\rho}{R}\right)^{\frac{p}{p-1}} \mathrm{Tail}\left((u-\mu^-)_- ; x_o,R;T_1,T_2\right)>\frac{1}{2}\sigma \xi \omega,
		\end{align*}
		or
		\begin{align*}
		|[u-\mu^-\leq\frac{1}{2}\sigma\xi\omega] \cap Q_{\rho}^-(\theta)| \leq \gamma \frac{\sigma^{p-1}}{\delta \alpha}|Q_{\rho}^-(\theta)|,
		\end{align*}
		where $\gamma$ only depends on $N,p,s,\Lambda$, the cylinder $Q_{2 \rho}^-(\theta)$ is included in $\mathcal{Q}$.
		\begin{proof}
			Let $(x_o,y_o)=(0,0)$ and $\mu^-=0$. We employ the Caccioppolic type inequality \eqref{3.1} for $(u-k)_-$ and $k=\sigma\xi\omega$ over the cylinder $B_{2\rho}\times(-\theta \rho^p, 0]$. Let the cutoff function $\zeta(x)\in [0,1]$ independent of time in $B_{2\rho}(x_o)$ vanishing on $\partial B_{\frac{3}{2}\rho}(x_o)$,	satisfy that $\zeta=1$ in $B_\rho(x_o)$ and $|\nabla \zeta|\leq\rho^{-1}$. We derive that
			\begin{align}
			\label{3.9}
			&\iint_{Q_{\rho}^-(\theta)} (u-k)_-(y,t)\int_{B_{2\rho}} \frac{(u-k)_+^{p-1}(x, t)}{|x-y|^{N+sp}}\,dxdyd\tau \nonumber\\
			\leq&\gamma\int_{-\theta\rho^p}^0\int_{B_{2\rho}}(u-k)_-^p(x,t)|\nabla\zeta(x,t)|^p\,dxd\tau+\int_{B_{2\rho}}(u-k)_-^2(x,-\theta\rho^p)\,dx\nonumber\\
			&+\gamma \int_{-\theta\rho^p}^0\int_{B_{2\rho}}\int_{B_{2\rho}}\max\{(u-k)_-(x,t), (u-k)_-(y,t)\}^p \frac{|\zeta(x)-\zeta(y)|^p}{|x-y|^{N+sp}}\,dxdyd\tau \nonumber\\
			&+\gamma\mathop{\mathrm{ess}\,\sup}_{\stackrel{-\theta\rho^p<t<0}{x\in B_{\frac{3}{2}\rho}}}\int_{\mathbb{R}^N \backslash B_{2\rho}}\frac{(u-k)_-^{p-1}(y,t)}{|x-y|^{N+sp}}\,dy\int_{-\theta\rho^p}^0\int_{B_{2\rho}}(u-k)_-\zeta^p(x,t)\,dxd\tau.
			\end{align}
			By enforcing
			\begin{align*}
			\left(\frac{\rho}{R}\right)^{\frac{p}{p-1}} \mathrm{Tail}\left(u_- ;R;T_1,T_2\right)\leq\frac{1}{2}\sigma \xi \omega
			\end{align*}
			and estimating the right-hand side of \eqref{3.9} as Lemma \ref{lem-3-2}, there holds that
			\begin{align}
			\label{3.10}
			\iint_{Q_{\rho}^-(\theta)} (u-k)_-(y,t)\int_{B_{2\rho}} \frac{(u-k)_+^{p-1}(x, t)}{|x-y|^{N+sp}}\,dxdyd\tau \leq\gamma \frac{(\sigma \xi \omega)^p}{\delta \rho^p}|Q_{\rho}^-(\theta)|.
			\end{align}
			Moreover, the left-hand side of \eqref{3.9} can be estimated by integral over smaller sets
			\begin{align}
			\label{3.11}
			&\iint_{Q_{\rho}^-(\theta)} (u-k)_-(y,t)\chi_{\left\{u(y,t) \leq \frac{1}{2} \sigma \xi \omega\right\}}\int_{B_{2\rho}} \frac{(u-k)_+^{p-1}(x,t) \chi_{\{u(x, t) \geq \xi \omega\}}}{|x-y|^{N+sp}}\,dxdydt\nonumber \\
			\geq&  \frac{1}{2} \sigma \xi \omega\left|\left[u \leq \frac{1}{2} \sigma \xi \omega\right] \cap Q_{\rho}^-(\theta)\right|\left(\frac{\left(\frac{1}{2} \xi \omega\right)^{p-1} \alpha\left|B_{\varrho}\right|}{(4 \rho)^{N+s p}}\right) \nonumber\\
			=&\gamma \frac{(\xi \omega)^p \alpha \sigma}{\rho^p}\left|\left[u \leq \frac{1}{2} \sigma \xi \omega\right]\cap Q_{\rho}^-(\theta)\right| .
			\end{align}
			Combining \eqref{3.10} and \eqref{3.11}, we finish the proof of this lemma.	
		\end{proof}
	\end{lemma}
	
	\section{Proof of Theorem 1.2 with $1<p\leq 2$}
	\label{sec4}
	
	The following proof in Sections 4 and 5 is analogous to that in \cite{L22}, but for the sake of completeness and readability, we present the details here. To streamline, we denote $Q_\rho(\theta)$ as the backward cylinders and omit the sign ``-". The definitions of $\mu^\pm$, $\omega$ and $\mathcal{Q}$ are the same as the ones in Section \ref{sec3}. 
	
	Next, we will devote to proving Theorem \ref{thm-1-2} in the case of $1<p\leq2$. We start this section with the expansion of positivity which plays an important role to get the reduction of oscillation in this scenario.
	
	\subsection{Expansion of positivity}
	\begin{proposition}
		\label{pro-4-1}
		Assume that $1<p\leq 2$. Let $u$ is a locally bounded, local weak supersolution to \eqref{1.1} in $E_T$. If for some constants $\alpha,\xi\in(0,1)$, it holds that
		\begin{align*}
		\left|[u(\cdot,t_o)-\mu^- \geq\xi\omega] \cap B_\rho(x_o)\right| \geq \alpha\left|B_\rho(x_o)\right|,
		\end{align*}
		then there exist $\eta,\delta\in(0,1)$ depending on $N,p,s,\Lambda$ and $\alpha$ such that either
		\begin{align*}
		\left(\frac{\rho}{R}\right)^{\frac{p}{p-1}} \mathrm{Tail}\left((u-\mu^-)_- ; x_o,R;T_1,T_2\right)>\eta \xi \omega,
		\end{align*}
		or
		\begin{align*}
		u-\mu^-\geq\eta \xi \omega\quad \text { a.e. in } B_{2 \rho}(x_o) \times\left(t_o+\frac{1}{2} \delta(\xi \omega)^{2-p} \rho^p, t_o+\delta(\xi \omega)^{2-p} \rho^p\right]
		\end{align*}
		with
		\begin{align*}
		B_{4\rho}(x_o) \times(t_o, t_o+\delta(\xi \omega)^{2-p} \rho^p] \subset \mathcal{Q}.
		\end{align*}
	\end{proposition}
	\begin{proof}
		Assume that $(x_o,t_o)=(0,0)$, $\mu^-=0$ and fix $\alpha\in(0,1)$. We can employ Lemma \ref{lem-3-4} in $B_{4\rho}$ since the assumption is about measure information at $t_o$. If we enforce
		\begin{align*}
		\left(\frac{\rho}{R}\right)^{\frac{p}{p-1}} \mathrm{Tail}\left(u_- ;R;T_1,T_2\right)\leq\xi \omega,
		\end{align*}
		there exist $\delta,\epsilon\in(0,1)$ obtained in Lemma \ref{lem-3-4} that only depend on $N,p,s,\Lambda$ and $\alpha$ such that
		\begin{align*}
		\left|[u(\cdot,t)>\epsilon \xi\omega] \cap B_{4\rho}\right| \geq \frac{\alpha}{2} 4^{-N} |B_{4\rho}| \quad \text { for all } t\in\left(0,\delta(\xi\omega)^{2-p}(4\rho)^p\right].
		\end{align*}
		The above estimate about each slice of time interval permits us to utilize Lemma \ref{lem-3-5} with $\xi,\alpha$ replaced by $\varepsilon\xi,\frac{1}{2}4^{-N}\alpha$.  Since $1<p\leq 2$ and $\sigma\in(0,1)$, it is easy to see
		\begin{align*}
		\delta(\sigma\epsilon\xi\omega)^{2-p}\leq\delta(\xi\omega)^{2-p}.
		\end{align*}
		Therefore, we can set $t_o=\delta(\sigma\epsilon\xi\omega)^{2-p}(4\rho)^p$ in Lemma \ref{lem-3-5}. If
		\begin{align*}
		\left(\frac{\rho}{R}\right)^{\frac{p}{p-1}} \mathrm{Tail}\left(u_- ;R;T_1,T_2\right)\leq\frac{1}{2}\sigma\epsilon\xi \omega,
		\end{align*}
		is in force, it holds that
		\begin{align*}
		|[u-\mu^-\leq\frac{1}{2}\sigma\epsilon\xi\omega] \cap Q_{4\rho}(\tilde{\theta})| \leq \gamma \frac{\sigma^{p-1}}{\delta \alpha}|Q_{4\rho}(\tilde{\theta})|,
		\end{align*}
		where $\tilde{\theta}=\delta(\sigma\epsilon\xi\omega)^{2-p}$. Then we let
		\begin{align*}
		\gamma \frac{\sigma^{p-1}}{\delta \alpha}<\nu, \quad \text { i.e., } \quad \sigma \leq\left(\frac{\nu \delta \alpha}{\gamma}\right)^{\frac{1}{p-1}},
		\end{align*}
		where $\nu$ is determined in Lemma \ref{lem-3-2} that depends on $N,p,s,\Lambda$ and $\alpha$.
		Given this choice of $\sigma$, if we enforce
		\begin{align*}
		\left(\frac{\rho}{R}\right)^{\frac{p}{p-1}} \operatorname{Tail}\left(u_{-} ; R;T_1,T_2\right) \leq \frac{1}{2}\sigma \epsilon \xi \omega,
		\end{align*}
		it follows from Lemma \ref{lem-3-2} with $t_o=\delta(\xi\omega)^{2-p}(4\rho)^p$ that
		\begin{align*}
		u \geq \frac{1}{4} \sigma \epsilon \xi \omega \quad \text { a.e. in } B_{2\rho} \times(\delta(\xi\omega)^{2-p}(4\rho)^p-\delta(\sigma \varepsilon \xi \omega)^{2-p}(2 \rho)^p, \delta(\xi \omega)^{2-p}(4 \rho)^p].
		\end{align*}
		The proof can be concluded by choosing $\eta=\frac{1}{4}\sigma\epsilon$.
	\end{proof}
	We emphasize that all constants determined in the proof are stable as $p\rightarrow 2$.
	\subsection{The first step of induction}
	In what follows, we study the reduction of oscillation by induction argument.
	For the cylinder $Q_{\widetilde{R}}\subset E_T$, we denote
	\begin{align*}
	\omega=2 \operatorname*{ess\,\sup}_{Q_{\widetilde{R}}}|u|+\mathrm{Tail}(u;x_o,\widetilde{R};t_o-\widetilde{R},t_o)
	\end{align*}
	and $Q_o=Q_R(\omega^{2-p})$. Obverse that we may let $Q_o\subset Q_{\widetilde{R}}$ by shrinking $R$ appropriately. Define
	\begin{align*}
	\mu^+= \operatorname*{ess\,\sup}_{Q_o} u,\quad \mu^-= \operatorname*{ess\,\inf}_{Q_o} u.
	\end{align*} 	
	We may assume $(x_o,t_o)=(0,0)$ by translation. Based on the definition above, we can see the intrinsic relation
	\begin{align}
	\label{4.1}
	\operatorname*{ess\,osc}_{Q_R(\omega^{2-p})} u\leq\omega.
	\end{align}
	
	Fix $\alpha=\frac{1}{2}$ and let $\delta\in(0,1)$ be chosen in Proposition \ref{pro-4-1}. Set $\tau=\delta(\frac{1}{4} \omega)^{2-p}(c R)^p$, where $c\in(0,\frac{1}{4})$ will be determined later. If $\mu^+-\mu^-\geq\frac{1}{2}\omega$, either
	\begin{align}
	\label{4.2}
	\left|\left[u(\cdot,-\tau)-\mu^->\frac{1}{4} \omega\right] \cap B_{cR}\right| \geq \frac{1}{2}|B_{cR}|,
	\end{align}
	or
	\begin{align}
	\label{4.3}
	\left|\left[\mu^+-u(\cdot,-\tau)>\frac{1}{4} \omega\right] \cap B_{cR}\right| \geq \frac{1}{2}|B_{cR}|
	\end{align}
	must hold. In addition, the other case $\mu^+-\mu^-\leq\frac{1}{2}\omega$ can directly reach the forthcoming estimate \eqref{4.6}. Here, we may assume \eqref{4.2} holds. As a consequence of Proposition \ref{pro-4-1} with $\alpha=\frac{1}{2},\xi=\frac{1}{4}$ and $\rho=cR$, there exists $\eta\in(0,\frac{1}{4})$ depending on $N,p,s,\Lambda$ such that either
	\begin{align}
	\label{4.4}
	c^{\frac{p}{p-1}} \mathrm{Tail}\left((u-\mu^-)_- ;R;-R^p\omega^{2-p},0\right)>\eta\omega,
	\end{align}
	or
	\begin{align}
	\label{4.5}
	u-\mu^-\geq\eta\omega  \quad \text { a.e. in } Q_{cR}\left(\delta\Big(\frac{1}{4}\omega\Big)^{2-p}\right).
	\end{align}
	In fact, if \eqref{4.4} is not true, from \eqref{4.5} and \eqref{4.1} it follows that
	\begin{align}
	\label{4.6}
	\operatorname*{ess\,osc}_{Q_{cR}(\delta(\frac{1}{4}\omega)^{2-p})} u\leq(1-\eta)\omega=:\omega_1.
	\end{align}
	
	In the following, we will devote to selecting the number $c$ to ensure \eqref{4.4} does not occur. We first estimate the tail term as
	\begin{align}
	\label{4.*}
	&[\mathrm{Tail}((u-\mu^-)_-;R;-R^p\omega^{2-p},0)]^{p-1}\nonumber\\
	=&R^p \operatorname*{ess\,\sup}_{-\omega^{2-p} R^p<t<0}\int_{\mathbb{R}^N \backslash B_R} \frac{(u-\mu^-)_-^{p-1}}{|x|^{N+sp}}\,dx\nonumber\\
	\leq& \gamma \omega^{p-1}+\gamma R^p \operatorname*{ess\,\sup}_{-\omega^{2-p} R^p<t<0} \int_{\mathbb{R}^N \backslash B_R} \frac{u_-^{p-1}}{|x|^{N+s p}}\,dx\nonumber\\
	=&\gamma \omega^{p-1}+\gamma R^p \operatorname*{ess\,\sup}_{-\omega^{2-p} R^p<t<0}\left[\int_{\mathbb{R}^N \backslash B_{\widetilde{R}}} \frac{u_-^{p-1}}{|x|^{N+s p}}\,dx+\int_{B_{\widetilde{R}} \backslash B_R} \frac{u_-^{p-1}}{|x|^{N+s p}}\,dx\right] \nonumber\\
	\leq& \gamma \omega^{p-1},
	\end{align}	
	where we used the definitions of $\omega$ and tail. Thus, we just need to take
	\begin{align}
	\label{4.7}
	c^{\frac{p}{p-1}} \gamma \omega \leq \eta \omega, \quad\text{ i.e., }\quad c \leq\left(\frac{\eta}{\gamma}\right)^{\frac{p-1}{p}},
	\end{align}
	then \eqref{4.4} will not happen.
	
	To verify the first step of induction is true, let $R_1=\lambda R$ for some $\lambda\leq c$ to satisfy
	\begin{align}
	\label{4.8}
	Q_{R_1}(\omega_1^{2-p}) \subset Q_{cR}(\delta(\frac{1}{4} \omega)^{2-p}), \quad \text { i.e., } \quad \lambda \leq 4^{\frac{p-2}{p}} \delta^{\frac{1}{p}} c ,
	\end{align}
	which together with \eqref{4.6} implies that
	\begin{align*}
	\operatorname*{ess\,osc}_{Q_{R_1}(\omega_1^{2-p})}u \leq \omega_1.
	\end{align*}
	\subsection{Induction}
	Next, we will proceed by a suitable induction to obtain the desired oscillation decay.
	For any $i=0,1,\cdots,j$, we define
	\begin{align*}
	R_o=R,\quad R_i=\lambda R_{i-1},\quad \omega_i=(1-\eta) \omega_{i-1},\quad Q_i=Q_{R_i}(\omega_i^{2-p})
	\end{align*}
	and
	\begin{align}
	\label{4.9}
	\mu_i^+= \operatorname*{ess\,\sup}_{Q_i} u,\quad \mu_i^-= \operatorname*{ess\,\inf}_{Q_i} u,\quad \operatorname*{ess\,osc}_{Q_i} u\leq \omega_i.
	\end{align}
	Assuming that the above oscillation estimate holds for some $j\geq 0$, we aim at proving it holds for $i=j+1$. Specifically, we can repeat the process of the first step to get the oscillation decay with $\mu_j^\pm, \omega_j, R_j, Q_j$ and
	$\tau=\delta(\frac{1}{4} \omega_j)^{2-p}(c R_j)^p$. If $\mu_j^+-\mu^-_j\geq\frac{1}{2}\omega_j$, one of the following two alternatives
	\begin{align*}
	\left|\left[u(\cdot,-\tau)-\mu^-_j>\frac{1}{4} \omega_j\right] \cap B_{cR_j}\right| \geq \frac{1}{2}|B_{cR_j}|,
	\end{align*}
	or
	\begin{align*}
	\left|\left[\mu_j^+-u(\cdot,-\tau)>\frac{1}{4} \omega_j\right] \cap B_{cR_j}\right| \geq \frac{1}{2}|B_{cR_j}|
	\end{align*}
	must hold.
	If $\mu_j^+-\mu^-_j\leq\frac{1}{2}\omega_j$, we can get the forthcoming estimate \eqref{4.11} immediately. Suppose the first alternative is true, which allows us to apply Proposition \ref{pro-4-1} in $Q_j$ with $\alpha=\frac{1}{2},\xi=\frac{1}{4}$ and $\rho=cR_j$. It follows that either
	\begin{align}
	\label{4.10}
	c^{\frac{p}{p-1}} \mathrm{Tail}((u-\mu_j^-)_- ; R_j;-R_j^p\omega_j^{2-p},0)>\eta \omega_j,
	\end{align}
	or
	\begin{align*}
	u-\mu^-_j\geq\eta\omega_j  \quad \text { a.e. in } Q_{cR_j}(\delta(\frac{1}{4}\omega_j)^{2-p}),
	\end{align*}
	which along with \eqref{4.9} tells
	\begin{align}
	\label{4.11}
	\operatorname*{ess\,osc}_{Q_{cR_j}(\delta(\frac{1}{4}\omega_j)^{2-p})} u\leq(1-\eta)\omega_j=:\omega_{j+1}.
	\end{align}
	
	Therefore, we need to choose the range of $c$ such that \eqref{4.10} does not happen. From the definition of tail, we have
	\begin{align}
	\label{4.12}
	&\quad [\mathrm{Tail}((u-\mu^-_j)_-;R_j;-R_j^p\omega_j^{2-p},0)]^{p-1}\nonumber \\
	&=R_j^p \operatorname*{ess\,\sup}_{-\omega_j^{2-p} R_j^p<t<0}\int_{\mathbb{R}^N \backslash B_j} \frac{(u-\mu^-_j)_-^{p-1}}{|x|^{N+sp}}\,dx\nonumber\\
	&=R_j^p\operatorname*{ess\,\sup}_{-\omega_j^{2-p} R_j^p<t<0}\left[\int_{\mathbb{R}^N \backslash B_R} \frac{(u-\mu_j^-)_-^{p-1}}{|x|^{N+s p}}\,dx+\sum_{i=1}^{j} \int_{B_{i-1} \backslash B_i} \frac{(u-\mu_j^-)_-^{p-1}}{|x|^{N+sp}}\,dx \right].
	\end{align}
	The first integral can be estimated as
	\begin{align}
	\label{4.13}
	\int_{\mathbb{R}^N \backslash B_R} \frac{(u-\mu_j^-)_-^{p-1}}{|x|^{N+sp}}\,dx & \leq \gamma \int_{\mathbb{R}^N \backslash B_R} \frac{|\mu_j^-|^{p-1}+u_-^{p-1}}{|x|^{N+sp}}\,dx \nonumber\\
	& \leq \gamma \frac{\omega^{p-1}}{R^{sp}}+\gamma \int_{B_{\widetilde{R}} \backslash B_R} \frac{u_-^{p-1}}{|x|^{N+sp}}\,dx+\gamma \int_{\mathbb{R}^N \backslash B_{\widetilde{R}}} \frac{u_-^{p-1}}{|x|^{N+sp}}\,dx \nonumber\\
	& \leq \gamma \frac{\omega^{p-1}}{R^p}
	\end{align}
	for every $t \in(-\omega_j^{2-p} R_j^p, 0)$, where we have used the definition of $\omega$. For the second integral, obverse that for $i=1,2, \cdots, j$,
	\begin{align*}
	(u-\mu_j^-)_- \leq \mu_j^--\mu_{i-1}^- \leq \mu_j^+-\mu_{i-1}^- \leq \mu_{i-1}^+-\mu_{i-1}^- \leq \omega_{i-1} \quad \text { a.e. in } Q_{i-1}.
	\end{align*}
	Thus, for every $t\in(-\omega_j^{2-p} R_j^p, 0)$,
	\begin{align}
	\label{4.14}
	\int_{B_{i-1} \backslash B_i} \frac{(u-\mu_j^-)_-^{p-1}}{|x|^{N+sp}}\,dx \leq \gamma \frac{\omega_{i-1}^{p-1}}{R_i^p}.
	\end{align}
	Taking \eqref{4.12}--\eqref{4.14} into account yields that
	\begin{align*}
	&\quad [\mathrm{Tail}((u-\mu^-_j)_-;R_j;-R_j^p\omega_j^{2-p},0)]^{p-1}\\ &\leq \gamma R_j^p \sum_{i=1}^{j} \frac{\omega_{i-1}^{p-1}}{R_i^p}\nonumber\\
	&=\gamma \omega_j^{p-1} \sum_{i=1}^{j}(1-\eta)^{(j-i+1)(1-p)} \lambda^{(j-i)p}.
	\end{align*}
	Indeed, if we take
	\begin{align*}
	(1-\eta)^{1-p} \lambda^p \leq \frac{1}{2}, \quad \text { i.e., } \quad \lambda \leq 2^{-\frac{1}{p}}(1-\eta)^{\frac{p-1}{p}},
	\end{align*}	
	there holds
	\begin{align*}
	[\mathrm{Tail}((u-\mu^-_j)_-;R_j;-R_j^p\omega_j^{2-p},0)]^{p-1}\leq \gamma\omega_j^{p-1}.
	\end{align*}
	Then \eqref{4.10} can be avoided if we choose constant $c$ as
	\begin{align}
	\label{4.15}
	c^p \gamma \leq \eta^{p-1}, \quad \text { i.e., } \quad c \leq \frac{1}{\gamma} \eta^{\frac{p-1}{p}} .
	\end{align}
	Here, we take the smaller of \eqref{4.7} and \eqref{4.15} as the final choice of $c$.
	
	As the first step, set  $R_{j+1}=\lambda R_j$ for some $\lambda \in(0,1)$ to satisfy
	\begin{align}
	\label{4.16}
	Q_{R_{j+1}}(\omega_{j+1}^{2-p}) \subset Q_{c R_j}\left(\delta(\frac{1}{4} \omega_j)^{2-p}\right), \quad \text { i.e., } \quad \lambda \leq 4^{\frac{p-2}{p}} \delta^{\frac{1}{p}} c .
	\end{align}
	Notice that due to the change of $c$, the choice of $\lambda$ may differ from the one in \eqref{4.8}. Hence, we choose $\lambda$ as
	\begin{align*}
	\lambda=\min \left\{2^{-\frac{1}{p}}(1-\eta)^{\frac{p-1}{p}}, 4^{\frac{p-2}{p}} \delta^{\frac{1}{p}} c\right\}.
	\end{align*}
	Then from \eqref{4.11} and \eqref{4.16}, we infer
	\begin{align*}
	\operatorname*{ess\,osc}_{Q_{R_{j+1}}(\omega_{j+1}^{2-p})} u \leq \omega_{j+1}.
	\end{align*}
	Since we get the reduction of oscillation, the arguments of proving H\"{o}lder regularity are standard, see \cite[Chapter III, Proposition 3.1]{D93}.
	
	\section{Proof of Theorem 1.2 with $p>2$}
	\label{sec5}
	In this section, for cylinders $Q_{\widetilde{R}}\subset E_T$, define
	\begin{align*}
	\omega=2 \operatorname*{ess\,\sup}_{Q_{\widetilde{R}}}|u|+\mathrm{Tail}(u;x_o,\widetilde{R};t_o-\widetilde{R}^p,t_o)
	\end{align*}
	and $Q_o=Q_R(a\theta)$ with $\theta=(\frac{1}{4}\omega)^{2-p}$, $a\in(0,1)$ that will be chosen later. By scaling $R$, we can let $Q_o\subset Q_{\widetilde{R}}$. Denote
	\begin{align*}
	\mu^+= \operatorname*{ess\,\sup}_{Q_o} u,\quad \mu^-= \operatorname*{ess\,\inf}_{Q_o} u.
	\end{align*}
	For simplicity, let $(x_o,t_o)=(0,0)$. It is easy to see
	\begin{align}
	\label{5.1}
	\operatorname*{ess\,osc}_{Q_R(a(\frac{1}{4}\omega)^{2-p})} u\leq\omega.
	\end{align}
	
	The case of $p>2$ is not the same as $1<p<2$ since the index in the definition of $\theta$ is negative.  Thus we need to choose suitable intrinsic cylinders for expansion of positivity and be careful in dealing with the tail term.
	
	\subsection{The first alternative}
	Next, we study the supersolution of \eqref{1.1} near its infimum. We discuss
	\begin{align}
	\label{5.2}
	\mu^+-\mu^->\frac{1}{2}\omega,
	\end{align}
	the other case can be processed similarly. Let $a,c\in(0,1)$ temporarily satisfy $a>2c^p$ that will be verified later. Suppose there holds that
	\begin{align}
	\label{5.3}
	\left|\left[u\leq \mu^-+\frac{1}{4} \omega\right] \cap(0, \bar{t})+Q_{c R}(\theta)\right| \leq \nu\left|Q_{c R}(\theta)\right|
	\end{align}
	for some $\bar{t} \in\left(-a \theta R^p+\theta(cR)^p, 0\right]$, where $\nu$ is the one given in Lemma \ref{lem-3-2}. Considering Lemma \ref{lem-3-2} with $\delta=1, \xi=\frac{1}{4}$ and $\rho=c R$ yields that either
	\begin{align}
	\label{5.4}
	c^{\frac{p}{p-1}} \mathrm{Tail}((u-\mu^-)_{-};R;-R^pa\theta,0)>\frac{1}{4}\omega,
	\end{align}
	or
	\begin{align}
	\label{5.5}
	u \geq \mu^-+\frac{1}{8}\omega \quad \text { a.e. in }(0, \bar{t})+Q_{\frac{1}{2} c R}(\theta).
	\end{align}
	In particular, the tail term can be processed as \eqref{4.*} to get
	\begin{align*}
	\mathrm{Tail}((u-\mu^-)_-;R;-R^pa\theta,0) \leq \gamma \omega.
	\end{align*}
	In order to avoid \eqref{5.4} happening, we choose $c$ as follows
	\begin{align}
	\label{5.6}
	c^{\frac{p}{p-1}} \gamma \omega \leq \frac{1}{4} \omega, \quad \text { i.e., } \quad c \leq\left(\frac{1}{4 \gamma}\right)^{\frac{p-1}{p}}.
	\end{align}
	By the pointwise estimate in \eqref{5.5}, we can apply Lemma \ref{lem-3-3} with $t_*=\bar{t}-\theta\left(\frac{1}{2} cR\right)^{p}$, $\rho=\frac{1}{2}cR$ to infer for some $\xi_o\in(0,\frac{1}{8})$ either
	\begin{align}
	\label{5.7}
	\left(\frac{c}{2} \right)^{\frac{p}{p-1}} \mathrm{Tail}((u-\mu^-)_-;R;-R^pa\theta,0)>\xi_{o} \omega,
	\end{align}
	or
	\begin{align}
	\label{5.8}
	u \geq \mu^-+\frac{1}{2} \xi_o \omega \quad \text{a.e. in}
	\quad B_{\frac{1}{4} c R} \times\left(t_*, t_*+\nu_o(\xi_{o} \omega)^{2-p}\left(\frac{1}{2} c R\right)^{p}\right].
	\end{align}
	Obverse that if we choose $\xi_o$ such that
	\begin{align}
	\label{5.9}
	\nu_o(\xi_o \omega)^{2-p}\left(\frac{1}{2}cR\right)^{p} \geq a\left(\frac{1}{4} \omega\right)^{2-p} R^{p}, \quad \text{i.e.,} \quad \xi_o=\frac{1}{4}\left(\frac{\nu_o c^{p}}{2^p a}\right)^{\frac{1}{p-2}},
	\end{align}
	the estimate \eqref{5.8} can be established up to $t=0$. We combine this with \eqref{5.1} to give that
	\begin{align}
	\label{5.10}
	\operatorname*{ess\,osc}_{Q_{\frac{1}{4}cR}(\theta)} u\leq\left(1-\frac{1}{2}\xi_o\right)\omega.
	\end{align}
	From \eqref{5.9}, the parameter $\xi_o$ is dependent on $c$ and $a$ which of both are not determined yet. Next, we assume $c$ is fixed and choose $a$ such that \eqref{5.7} does not occur. Since the tail term can be evaluated as before and be bounded by $\gamma\omega$, we take
	\begin{align*}
	\left(\frac{c}{2} \right)^{\frac{p}{p-1}} \gamma \omega \leq \xi_o\omega \equiv \frac{1}{4}\left(\frac{\nu_o c^p}{2^p a}\right)^{\frac{1}{p-2}} \omega.
	\end{align*}
	Thus the relation between $c$ and $a$ can be obtained from the above inequality as
	\begin{align}
	\label{5.11}
	a=\frac{\nu_o}{\gamma} c^{\frac{p}{p-1}}.
	\end{align}
	Consequently, we can see \eqref{5.10} actually holds by setting $a$ as \eqref{5.11}. Moreover, taking $c$ as
	\begin{align}
	\label{5.12}
	a=\frac{\nu_o}{\gamma} c^{\frac{p}{p-1}}>2 c^p, \quad \text { i.e., } \quad c<\left(\frac{\nu_o}{\gamma}\right)^{\frac{p-1}{p(p-2)}},
	\end{align}
	we can verify the assumption $a>2 c^p$ at the beginning. In fact, if we choose $c$ as \eqref{5.12}, the number $a$ can be determined via \eqref{5.11} finally.
	
	\subsection{The second alternative}
	
	In the following, we study the subsolution of \eqref{1.1} near its supremum. If \eqref{5.3} does not hold for every $\bar{t} \in\left(-a \theta R^p+\theta(cR)^p, 0\right]$,
	we deduce that
	\begin{align}
	\label{5.13}
	\left|\left[\mu^+-u\geq\frac{1}{4} \omega\right] \cap(0, \bar{t})+Q_{c R}(\theta)\right| > \nu\left|Q_{c R}(\theta)\right|.
	\end{align}
	From this, there exists some $t_* \in[\bar{t}-\theta(cR)^p, \bar{t}-\frac{1}{2} \nu \theta(c R)^p]$ such that
	\begin{align}
	\label{5.14}
	\left|\left[\mu^+-u(\cdot,t_*)\geq\frac{1}{4} \omega\right] \cap B_{cR}\right|>\frac{1}{2} \nu|B_{c R}|.
	\end{align}
	In fact, if \eqref{5.14} is false for any $s \in[\bar{t}-\theta(cR)^p, \bar{t}-\frac{1}{2} \nu \theta(c R)^p]$, we have
	\begin{align*}
	&\quad \left|\left[\mu^+-u\geq\frac{1}{4} \omega\right] \cap(0, \bar{t})+Q_{c R}(\theta)\right| \\
	&=\int_{\bar{t}-\theta(cR)^p}^{\bar{t}-\frac{1}{2} \nu \theta(cR)^p}\left|\left[\mu^+-u(\cdot, s) \geq \frac{1}{4} \omega\right] \cap B_{cR}\right|\,ds \\
	&\quad+\int_{\bar{t}-\frac{1}{2} \nu \theta(cR)^p}^{\bar{t}}\left|\left[\mu^+-u(\cdot, s) \geq \frac{1}{4} \omega\right] \cap B_{cR}\right|\,ds \\
	&<\frac{1}{2} \nu|B_{cR}| \theta(c R)^p\left(1-\frac{1}{2} \nu\right)+\frac{1}{2} \nu \theta(c R)^p|B_{cR}| \\
	&<\nu|Q_{cR}(\theta)|,
	\end{align*}
	which contradicts \eqref{5.13}. Thus the result \eqref{5.14} allows us to apply Lemma \ref{lem-3-4} with $\alpha=\frac{1}{2} \nu, \rho=cR$ and the parameter $\xi_1\in(0,\frac{1}{4})$, then there exist $\delta$ and $\epsilon$ depending on $N,p,s,\Lambda$ and $\nu$ such that either
	\begin{align}
	\label{5.15}
	c^{\frac{p}{p-1}} \mathrm{Tail}((\mu^+-u)_+; R;-R^pa\theta,0)>\xi_1 \omega,
	\end{align}
	or
	\begin{align}
	\label{5.16}
	\left|\left[\mu^+-u(\cdot, t) \geq \epsilon \xi_1\omega\right] \cap B_{cR}\right| \geq \frac{\alpha}{2}|B_{cR}| \quad \text{for all} \quad t \in\left(t_*, t_*+\delta(\xi_1 \omega)^{2-p}(cR)^p\right].
	\end{align}
	By choosing
	\begin{align*}
	\delta(\xi_1 \omega)^{2-p}(cR)^p\geq \theta(cR)^p \quad\text { i.e., } \quad \xi_1=\frac{1}{4} \delta^{\frac{1}{p-2}},
	\end{align*}
	it follows that the estimate \eqref{5.16} can be justified up to the time level $\bar{t}$. Due to the arbitrariness of $\bar{t}$, there holds
	\begin{align}
	\label{5.17}
	\left|\left[\mu^+-u(\cdot,t) \geq\epsilon \xi_1 \omega\right] \cap B_{cR}\right| \geq \frac{\alpha}{2}|B_{c R}|\quad\text{for all } t \in(-a \theta R^p+\theta(c R)^p, 0].
	\end{align}
	If \eqref{5.15} does not occur,  we need restrict $c$ as
	\begin{align}
	\label{5.18}
	c^{\frac{p}{p-1}}\gamma\omega\leq\xi_1\omega\quad\text{i.e.,}\quad c \leq\left(\frac{\xi_1}{\gamma}\right)^{\frac{p-1}{p}}.
	\end{align}
	Moreover, if
	\begin{align}
	\label{5.19}
	a \theta R^p-\theta(cR)^p\geq(\sigma \epsilon \xi_1 \omega)^{2-p}(c R)^p, \quad\text{i.e.,} \quad c \leq(\frac{\nu_o}{\gamma})^{\frac{p-1}{p(p-2)}}(\sigma \epsilon \xi_1)^{\frac{p-1}{p}},
	\end{align}
	\eqref{5.17} leads to
	\begin{align*}
	\left|\left[\mu^+-u(\cdot,t) \geq \epsilon \xi_1 \omega\right] \cap B_{cR}\right| \geq \frac{\alpha}{2}|B_{c R}|\quad\text{for all}\quad  t \in(-(\sigma \epsilon \xi_1 \omega)^{2-p}(c R)^p, 0].
	\end{align*}
	Hence, we can employ Lemma \ref{lem-3-5} with $\delta=1, \xi=\epsilon \xi_1$ and $\rho=cR$ to get either
	\begin{align}
	\label{5.20}
	c^{\frac{p}{p-1}} \mathrm{Tail}((\mu^+-u)_+ ; R;-R^pa\theta,0)>\frac{1}{2}\sigma \epsilon \xi_1 \omega,
	\end{align}
	or
	\begin{align*}
	\left|\left[\mu^+-u \leq\frac{1}{2} \sigma \epsilon \xi_1 \omega\right]\cap Q_{c R}(\widetilde{\theta})\right| \leq \gamma\frac{\sigma^{p-1}}{\alpha}|Q_{cR}(\widetilde{\theta})|,
	\end{align*}
	where $\widetilde{\theta}=(\sigma \epsilon \xi_1 \omega)^{2-p}$ and $\gamma$ depends only on $N,p,s,\Lambda$. After that, we fix $\nu$ as in Lemma \ref{lem-3-2} and choose $\sigma\in(0,\frac{1}{2})$ satisfying
	\begin{align*}
	\gamma \frac{\sigma^{p-1}}{\alpha} \leq \nu,
	\end{align*}
	that permits us to use Lemma \ref{lem-3-2}, it follows that
	\begin{align*}
	\mu^+-u\geq\frac{1}{4}\sigma\epsilon\xi_1\omega\quad \text { a.e. in }  Q_{\frac{1}{2} c R}(\widetilde{\theta}).
	\end{align*}
	This and \eqref{5.1} ensure
	\begin{align}
	\label{5.21}
	\operatorname*{ess\,osc}_{Q_{\frac{1}{2}cR}(\widetilde{\theta})} u\leq\left(1-\frac{1}{4}\sigma\epsilon\xi_1\right)\omega.
	\end{align}
	Here, we also need choose $c$ such that \eqref{5.20} does not occur, that is
	\begin{align}
	\label{5.22}
	c \leq\left(\frac{\sigma \epsilon \xi_1}{\gamma}\right)^{\frac{p-1}{p}}.
	\end{align}
	In view of \eqref{5.10} and \eqref{5.21}, we obtain
	\begin{align}
	\label{5.23}
	\operatorname*{ess\,osc}_{Q_{\frac{1}{4}cR}(\theta)} u\leq(1-\eta)\omega:=\omega_1,
	\end{align}
	where
	\begin{align*}
	\eta=\min \left\{\frac{1}{2} \xi_o, \frac{1}{4} \sigma \epsilon \xi_{1}\right\}.
	\end{align*}
	
	In order to get the first step of induction, set $\theta_1=\left(\frac{1}{4} \omega_1\right)^{2-p}$ and $R_1=\frac{1}{4}cR$. We also let
	\begin{align*}
	Q_{R_1}(a \theta_1) \subset Q_{\frac{1}{4} c R}(\theta), \quad \text { i.e., } \quad a \leq(1-\eta)^{p-2}.
	\end{align*}
	Recalling the choice of $a$ in \eqref{5.11}, the number $c$ should be restricted as
	\begin{align}
	\label{5.24}
	c\leq\left[\left(1-\eta\right)^{p-2}\frac{\gamma}{\nu_o}\right]^{\frac{p-1}{p}}.
	\end{align}
	We now derive the oscillation estimate
	\begin{align}
	\label{5.25}
	\operatorname*{ess\,osc}_{Q_{R_1}(a\theta_1)} u\leq\omega_1.
	\end{align}	
	\subsection{Induction}
	Once get \eqref{5.25}, we start to proceed by induction. For any $i=0,1,\cdots,j$, define
	\begin{align*}
	R_o=R,\quad R_i=\frac{1}{4}cR_{i-1},\quad \omega_i=(1-\eta) w_{i-1},\quad\theta_i=\left(\frac{1}{4}\omega_i\right)^{2-p},\quad Q_i=Q_{R_i}(a\theta_i)
	\end{align*}
	and
	\begin{align*}
	\mu_i^+= \operatorname*{ess\,\sup}_{Q_i} u,\quad \mu_i^-= \operatorname*{ess\,\inf}_{Q_i} u,\quad \operatorname*{ess\,osc}_{Q_i} u\leq w_i.
	\end{align*}
	We suppose the above oscillation is true for some $j\geq 0$ and then prove it still hold for $i=j+1$.
	
	As the proof in Section \ref{sec4}, we repeat the previous process with $\mu_j^\pm, \omega_j, R_j, \theta_j, Q_j$ and arrive at
	\begin{align}
	\label{5.26}
	\operatorname*{ess\,osc}_{Q_{\frac{1}{4}cR_j}(\theta_j)} u\leq(1-\eta)\omega_j:=\omega_{j+1}.
	\end{align}
	Let $\theta_{j+1}=(\frac{1}{4} \omega_{j+1})^{2-p}$ and $R_{j+1}=\frac{1}{4} c R_j$. By the choice of  $a\leq(1-\eta)^{p-2}$, we can directly verify
	\begin{align*}
	Q_{R_{j+1}}(a \theta_{j+1}) \subset Q_{\frac{1}{4} c R_j}(\theta_j),
	\end{align*}
	which together with \eqref{5.26} leads to
	\begin{align*}
	\operatorname*{ess\,osc}_{Q_{R_{j+1}}(a\theta_{j+1})} u\leq\omega_{j+1}.
	\end{align*}
	To ensure the oscillation decay \eqref{5.26}, we choose $c$ from \eqref{5.6},\eqref{5.12},\eqref{5.18},\eqref{5.22} and \eqref{5.24} such that the alternative involving the tail does not happen. In the induction process, we again estimate tail as Section \ref{sec4} to obtain
	\begin{align*}
	\mathrm{Tail}\left(\left(\mu_j^+-u\right)_+ ;R_j;-R_j^pa\theta,0\right) \leq \gamma \omega_{j},
	\end{align*}
	and omit the details here. Taking into account the previous estimates, we can determine $c$ which is independent of $j$. Then, by \eqref{5.11} the number $a$ can be chosen finally. Hence, we get the reduction of oscillation which implies the H\"{o}lder continuity of weak subsolutions.

	\section*{Acknowledgments}
	The authors wish to thank the anonymous reviewer for valuable comments and suggestions. This work was supported by the National Natural Science Foundation of China (No. 12071098).

\end{document}